\documentclass{amsart}
\usepackage{amsmath}
\usepackage{amssymb}
\usepackage{amsfonts}
\usepackage{hyperref}

\newtheorem{theorem}{Theorem}
\theoremstyle{plain}

\newtheorem{corollary}{Corollary}

\newtheorem{definition}{Definition}

\newtheorem{lemma}{Lemma}

\newtheorem{remark}{Remark}

\newtheorem{theoremalph}{Theorem}

\numberwithin{equation}{section}

\begin{document}
\title[Some weighted  Ostrowski type inequalities on time scales]{Some weighted  Ostrowski type inequalities on time scales involving combination of weighted $\Delta$-integral  means}
\author[W. J. Liu]{Wenjun Liu}
\address[W. J. Liu]{College of Mathematics and Statistics\\
Nanjing University of Information Science and Technology \\
Nanjing 210044, China}
\email{wjliu@nuist.edu.cn}
\author[H. R\"{u}zgar]{H\"{u}seyin R\"{u}zgar}
\address[H. R\"{u}zgar]{Department of Mathematics\\
Faculty of Science and Arts\\
University of Ni\u{g}de\\
Merkez 51240, Ni\u{g}de, Turkey}
\email{091908002@nigde.edu.tr}

\author[A. Tuna]{Adnan Tuna}
\address[A. Tuna]{Department of Mathematics\\
Faculty of Science and Arts\\
University of Ni\u{g}de\\
Merkez 51240, Ni\u{g}de, Turkey}
\email{atuna@nigde.edu.tr}

\begin{abstract}
In this paper we obtain some weighted generalizations of Ostrowski type inequalities on time scales involving combination of weighted $\Delta$-integral  means, i.e., a weighted Ostrowski type inequality on time scales involving combination of weighted $\Delta$-integral  means, two weighted Ostrowski type inequalities for two functions on time scales,  and some weighted perturbed Ostrowski type inequalities on
time scales. We also give some other interesting inequalities and recapture some known results as special cases.
\end{abstract}

\maketitle

\section{Introduction}

In 1937, Ostrowski derived a formula to estimate the absolute deviation of a
differentiable function from its integral mean \cite{20}. The result is
nowadays known as the Ostrowski inequality \cite{fahmad, d20011, d20012, Tseng}, which can
be obtained by using the Montgomery identity. Recently, Ahmad et. al \cite{fahmad} developed some new Ostrowski   inequalities involving two functions, by using an identity
of Dragomir and Barnett proved in \cite{db1999}. In \cite{Tseng}, Tseng,
Hwang and Dragomir established the following generalizations of weighted
Ostrowski type inequalities for mappings of bounded variation.

\begin{theoremalph}
Let us have $0\leq \alpha \leq 1,$ $g:[a,b]\rightarrow \mathbf{[}0\mathbf{%
,\infty )}$ continuous and positive on $(a,b)$ and let $h:[a,b]\rightarrow
\mathbb{R}$ be differentiable such that $h^{\prime }(t)=g(t)$ on $[a,b].$
Let $c=h^{-1}\left( \left( 1-\frac{\alpha }{2}\right) h(a)+\frac{\alpha }{2}%
h(b)\right) $ and $d=h^{-1}\left( \frac{\alpha }{2}h(a)+\left( 1-\frac{%
\alpha }{2}\right) h(b)\right) $. Let $f:[a,b]\rightarrow \mathbb{R}$ be a
mapping of bounded variation. Then, for all $x\in \left[ c,d\right] ,$ we
have
\begin{equation}
\left\vert \int_{a}^{b}f(t)g(t)dt-\left[ (1-\alpha )f(x)+\alpha \frac{%
f(a)+f(b)}{2}\right] \int_{a}^{b}g(t)dt\right\vert \leq K\bigvee_{a}^{b}(f),
\label{eq1.3}
\end{equation}%
where
\begin{equation}
K:=\left\{
\begin{array}{ll}
\displaystyle\frac{1-\alpha }{2}\int_{a}^{b}g(t)dt+\left\vert h(x)-\frac{%
h(a)+h(b)}{2}\right\vert , & 0\leq \alpha \leq \frac{1}{2}, \\
\displaystyle\max \left\{ \frac{1-\alpha }{2}\int_{a}^{b}g(t)dt+\left\vert
h(x)-\frac{h(a)+h(b)}{2}\right\vert ,\frac{\alpha }{2}\int_{a}^{b}g(t)dt%
\right\} , & \frac{1}{2}<\alpha <\frac{2}{3}, \\
\displaystyle\frac{\alpha }{2}\int_{a}^{b}g(t)dt, & \frac{2}{3}\leq \alpha
\leq 1%
\end{array}%
\right.  \label{eq1.4}
\end{equation}%
and $\bigvee\limits_{a}^{b}(f)$ denotes the total varfiation of $f$ on the
interval $\left[ a,b\right] .$ In (\ref{2}), the constant $\frac{1-\alpha }{2%
}$ for $0\leq \alpha \leq \frac{1}{2}$ and the constant $\frac{\alpha }{2}$
for $\frac{2}{3}\leq \alpha \leq 1$ are the best possible.
\end{theoremalph}

 In 1988, Hilger introduced the time scale theory in order to unify
continuous and discrete analysis \cite{9}. Such theory has a tremendous
potential for applications in some mathematical models of real processes and
phenomena studied in population dynamics \cite{bfz}, economics \cite{abl},
physics \cite{spc}, space weather \cite{s3} and so on. Recently, many
authors studied the theory of certain integral inequalities on time scales
(see \cite{4,5,8',hla2011,l2010,ls,12,14,15,16,19,20',s1,21,22}).
The Ostrowski inequality and
the Montgomery identity were generalized by Bohner and Matthews to an
arbitrary time scale \cite{5}, unifying the discrete, the continuous, and
the quantum cases:

\begin{theoremalph}[Ostrowski's inequality on time scales \protect\cite{5}]
\label{thm:1} Let $a,b,s,t\in \mathbb{T},$ $\ a<b$ and $f:[a,b]\rightarrow
\mathbb{R}$ be differentiable. Then%
\begin{equation}
\left\vert f(t)-\frac{1}{b-a}\int_{a}^{b}f(\sigma (s))\Delta s\right\vert
\leq \frac{M}{b-a}[h_{2}(t,a)+h_{2}(t,b)],  \label{1}
\end{equation}
where $h_2(\cdot, \cdot)$ is defined by Definition \ref{de8} below and $%
M=\sup\limits_{a< t<b}\left\vert f^{\Delta }(t)\right\vert <\infty . $ This
inequality is sharp in the sense that the right-hand side of $(\ref{1})$
cannot be replaced by a smaller one.
\end{theoremalph}

Very recently, the authors \cite{jiang} gave some new generalizations of Ostrowski type
inequalities on time scales involving combination of $\Delta$-integral  means by using the kernel given in \cite{8}.
The purpose of this paper is to obtain some weighted  Ostrowski type inequalities on time scales involving combination of weighted $\Delta$-integral  means. We first establish a weighted Ostrowski type inequality on time scales involving combination of weighted $\Delta$-integral  means. Then we derive two weighted Ostrowski type inequalities for two functions on time scales. Finally, four  weighted perturbed Ostrowski type inequalities on
time scales are established. We also give some other interesting inequalities and recapture some known results as special cases.

This paper is organized as follows. In Section \ref{s2}, we briefly present
the general definitions and theorems related to the time scales calculus. Some
weighted Ostrowski type inequalities on time scales are derived in Section %
\ref{s3}.

\section{Time Scales Essentials}

\label{s2}

\vspace{0cm} In this section we briefly introduce the time scales theory.
For further details and proofs we refer the reader to Hilger's Ph.D. thesis
\cite{9}, the books \cite{2,3,11}, and the survey \cite{1}.

\begin{definition}
A time scale $\mathbb{T}$ is an arbitrary nonempty closed subset of $\mathbb{%
\
\mathbb{R}
}$. For $t\in \mathbb{T}$, we define the forward jump operator $\sigma :%
\mathbb{T\rightarrow T}$ by $\sigma (t)=\inf \left\{ s\in \mathbb{T}
:s>t\right\} ,$ while the backward jump operator $\rho :\mathbb{T\rightarrow
T}$ is defined by $\rho (t)=\sup \left\{ s\in \mathbb{T}:\ s<t\right\} $.
The jump operators $\sigma $ and $\rho $ allow the classification of points
in $\mathbb{T}$ as follows. If $\sigma (t)>t,$\textit{\ then we say that }$t$
is right-scattered, while if $\rho (t)<t$ then we say that $t$ is
left-scattered. Points that are right-scattered and left-scattered at the
same time are called isolated. If $\sigma (t)=t,$ the $t$ is called
right-dense, and if $\rho (t)=t$ then $t$ is called left-dense, Points that
both right-dense and left-dense are called dense. The mapping $\mu :\mathbb{%
\ T\rightarrow
\mathbb{R}
}^{+}$ defined by $\mu (t)=\sigma (t)-t$ is called the \textit{graininess
function. }The set $\mathbb{T}^{k}$ is defined as follows: if $\ \mathbb{T}$
has a left-scattered maximum $m,$ then $\mathbb{T}^{k}=\mathbb{T}-\left\{
m\right\} ;$ otherwise, $\mathbb{T}^{k}=\mathbb{T}$.
\end{definition}

If $\mathbb{T=\mathbb{R}}$, then $\mu (t)=0,$ and when $\mathbb{T}= \mathbb{Z%
},$ we have $\mu (t)=1.$

\begin{definition}
Let $f:\mathbb{T\rightarrow \mathbb{R} }.$ $f$ is called differentiable at $%
t\in \mathbb{T}^{k},$ with (delta) derivative $f^{\Delta }(t)\in \mathbb{\
\mathbb{R}}$, if for any given $\varepsilon >0$ there exists a neighborhood $%
U$ of $t$ such that
\begin{equation*}
\left\vert f (\sigma(t))-f(s)-f^{\Delta }(t)[\sigma (t)-s]\right\vert \leq
\varepsilon \left\vert \sigma (t)-s\right\vert,\quad \forall\ s\in U.
\end{equation*}
\end{definition}

If $\mathbb{T=R}$, then$~f^{\Delta }(t)=\frac{df(t)}{dt},$ and if $\mathbb{\
T=Z}$, then $f^{\Delta }(t)=f(t+1)-f(t).$

\begin{theoremalph}
\label{th3} Assume $f,g:\mathbb{T\rightarrow \mathbb{R} }$ are
differentiable at $t\in \mathbb{T}^{k}.$ Then the product $fg:\mathbb{\
T\rightarrow \mathbb{R} }$ is differentiable at $t$ with
\begin{align*}
(fg)^{\Delta }(t) =f^{\Delta }(t)g(t)+f(\sigma (t))g^{\Delta }(t)
=f(t)g^{\Delta }(t)+f^{\Delta }(t)g(\sigma (t)).
\end{align*}
\end{theoremalph}

\begin{definition}
The function $f:\mathbb{T\rightarrow \mathbb{R} }$ is said to be \textit{\
rd-continuous (denote }$f\in C_{rd}(\mathbb{T}$,$\mathbb{\ \mathbb{R}})$),
if it is continuous at all right-dense points $t\in \mathbb{T}$ and its
left-sided limits exist at all left-dense points $t\in \mathbb{T}$.
\end{definition}

It follows from \cite[Theorem 1.74]{2} that every rd-continuous function has
an anti-derivative.

\begin{definition}
Let $f\in C_{rd}(\mathbb{T}$,$\mathbb{\mathbb{R} }).$ Then $F:\mathbb{\
T\rightarrow \mathbb{R} }$ is called the antiderivative of $f$ on $\mathbb{\
T\ }$ if it satisfies $F^{\Delta }(t)=f(t)$ \ for any$\ t\in \mathbb{T}^{k}$
. In this case, we define the $\Delta$-integral of $f$ as
\begin{equation*}
\int_{a}^{t}f(s)\Delta s=F(t)-F(a),\ \ t\in \mathbb{T}\text{.}
\end{equation*}
\end{definition}

\begin{theoremalph}
\label{th4} Let $f,g$ be rd-continuous, $a,b,c\in \mathbb{T}$ and $\alpha
,\beta \in \mathbb{R}$. Then

$(1)$ $\int_{a}^{b}\left[\alpha f(t)+\beta g(t)\right] \Delta
t=\alpha\int_{a}^{b}f(t)\Delta t+\beta\int_{a}^{b}g(t)\Delta t,$

$(2)$ $\int_{a}^{b}f(t)\Delta t=-\int_{b}^{a}f(t)\Delta t,$

$(3)$ $\int_{a}^{b}f(t)\Delta t=\int_{a}^{c}f(t)\Delta
t+\int_{c}^{b}f(t)\Delta t,$

$(4)$ $\int_{a}^{b}f(t)g^{\Delta }(t)\Delta
t=(fg)(b)-(fg)(a)-\int_{a}^{b}f^{\Delta }(t)g(\sigma (t))\Delta t,$
\end{theoremalph}

\begin{theoremalph}
If $f$ is $\Delta $-integrable on $[a,b]$, then so is $\left\vert
f\right\vert ,$and
\begin{equation*}
\left\vert \int_{a}^{b}f(t)\Delta t\right\vert \leq \int_{a}^{b}\left\vert
f(t)\right\vert \Delta t.
\end{equation*}
\end{theoremalph}

\begin{definition}
\label{de8} Let $h_{k}:\mathbb{T}^{2}\rightarrow \mathbb{R}$, $k\in \mathbb{N%
}_{0}$ be defined by
\begin{equation*}
h_{0}(t,s)=1\text{ \ \ for all \ \ }s,t\in \mathbb{T}
\end{equation*}
and then recursively by
\begin{equation*}
h_{k+1}(t,s)=\int_{s}^{t}h_{k}(\tau, s)\Delta \tau\text{ \ \ \ for all \ \ }
s,t\in \mathbb{T}.
\end{equation*}
\end{definition}

\section{Main Results}

\label{s3}

\subsection{A weighted Ostrowski type inequality on time scales}

We first establish a weighted Ostrowski type inequality on time scales involving combination of weighted $\Delta$-integral  means. For this purpose, we need the following lemma:

\begin{lemma}
\label{lm1}Let $a,b,s,t\in \mathbb{T},$ $\ a<b$ and $f,h:[a,b]\rightarrow
\mathbb{R}$ be differentiable. Then for all $x\in \lbrack a,b]$, we have the weighted Montgomery identity on time scales involving combination of weighted $\Delta$-integral  means
\begin{align}
\int_{a}^{b}P(x,t)f^{\Delta }(t)\Delta t =&\frac{f(x)}{\alpha +\beta }\left[
\alpha \frac{h(x)-h(a)}{x-a}+\beta \frac{h(b)-h(x)}{b-x}\right]  \notag \\
&-\frac{1}{\alpha +\beta }\left[ \frac{\alpha }{x-a}\int_{a}^{x}h^{\Delta
}(t)f(\sigma \left( t\right) )\Delta t+\frac{\beta }{b-x}\int_{x}^{b}h^{%
\Delta }(t)f(\sigma \left( t\right) )\Delta t\right],  \label{2}
\end{align}%
where $\alpha, \beta\in \mathbb{R}$ are nonnegative and not both zero,
\begin{equation}
P\left( x,t\right) =\left\{
\begin{array}{l}
\frac{\alpha }{\alpha +\beta }\left( \frac{h(t)-h(a)}{x-a}\right), \text{ \
\ }a\leq t<x, \medskip\\
\frac{-\beta }{\alpha +\beta }\left( \frac{h(b)-h(t)}{b-x}\right), \text{ \ \ }%
x\leq t<b,%
\end{array}%
\right.  \label{3}
\end{equation}%
which is the weighted version of the kernal given in \cite{8}.
\end{lemma}

\begin{proof}
Using Theorem \ref{th4} (4), we have%
\begin{align}
&\int_{a}^{x}\frac{\alpha }{\alpha +\beta }\left( \frac{h(t)-h(a)}{x-a}%
\right) f^{\Delta }(t)\Delta t  \notag \\
=&\frac{\alpha }{\alpha +\beta }\left( \frac{h(x)-h(a)}{x-a}\right) f\left(
x\right) -\frac{\alpha }{\left( \alpha +\beta \right) \left( x-a\right) }%
\int_{a}^{x}h^{\Delta }(t)f(\sigma \left( t\right) )\Delta t  \label{4}
\end{align}%
and%
\begin{align}
&\int_{x}^{b}\frac{-\beta }{\alpha +\beta }\left( \frac{h(b)-h(t)}{b-x}%
\right) f^{\Delta }(t)\Delta t  \notag \\
=&\frac{\beta }{\alpha +\beta }\left( \frac{h(b)-h(x)}{b-x}\right) f\left(
x\right) -\frac{\beta }{\left( \alpha +\beta \right) \left( b-x\right) }%
\int_{x}^{b}h^{\Delta }(t)f(\sigma \left( t\right) )\Delta t.  \label{5}
\end{align}%
Therefore, the identitiy (\ref{2}) is proved by adding the above two
identities.
\end{proof}

\begin{remark}
If we take $h(t)=t$ and $\mathbb{T}=\mathbb{R}$ in Lemma \ref{lm1},   we
obtain the identity given in \cite[Lemma 1]{8}.
\end{remark}

\begin{corollary}
If we take $\mathbb{T}=\mathbb{R}$ in Lemma \ref{lm1}, then we get the weighted Montgomery identity
\begin{align*}
\int_{a}^{b}P(x,t)f^{\prime }(t)dt =&\frac{f(x)}{\alpha +\beta }\left[
\alpha \frac{h(x)-h(a)}{x-a}+\beta \frac{h(b)-h(x)}{b-x}\right] \\
&-\frac{1}{\alpha +\beta }\left[ \frac{\alpha }{x-a}\int_{a}^{x}h^{\prime
}(t)f(t)dt+\frac{\beta }{b-x}\int_{x}^{b}h^{\prime }(t)f(t)dt\right],
\end{align*}%
where%
\begin{equation*}
P\left( x,t\right) =\left\{
\begin{array}{l}
\frac{\alpha }{\alpha +\beta }\left( \frac{h(t)-h(a)}{x-a}\right), \text{ \
\ }a\leq t<x, \medskip\\
\frac{-\beta }{\alpha +\beta }\left( \frac{h(b)-h(t)}{b-x}\right), \text{ \ \ }%
x\leq t<b.
\end{array}%
\right.
\end{equation*}
\end{corollary}

\begin{corollary}
If we take $\mathbb{T}=\mathbb{Z}$ in Lemma \ref{lm1}, then we get
\begin{align*}
\sum\limits_{t=a}^{b-1}P(x,t)\Delta f(t) =&\frac{f(x)}{\alpha +\beta }\left[
\alpha \frac{h(x)-h(a)}{x-a}+\beta \frac{h(b)-h(x)}{b-x}\right] \\
&-\frac{1}{\alpha +\beta }\left[ \frac{\alpha }{x-a}\sum%
\limits_{t=a}^{x-1}f(t+1)\Delta h(t)+\frac{\beta }{b-x}\sum%
\limits_{t=x}^{b-1}f(t+1)\Delta h(t)\right],
\end{align*}%
where%
\begin{equation*}
P\left( x,t\right) =\left\{
\begin{array}{l}
\frac{\alpha }{\alpha +\beta }\left( \frac{h(t)-h(a)}{x-a}\right) ,\text{ }%
a\leq t<x-1, \medskip\\
\frac{-\beta }{\alpha +\beta }\left( \frac{h(b)-h(t)}{b-x}\right) ,\text{ }%
x\leq t<b-1.%
\end{array}%
\right.
\end{equation*}
\end{corollary}

\begin{corollary}
If we take $\mathbb{T}=q^{\mathbb{Z}}\cup \{0\}\ (q>1)$ in Lemma \ref{lm1},
then we get%
\begin{align*}
\int_{a}^{b}P(x,t)\mathrm{D}_{q}f(t)\mathrm{d}_{q}t =&\frac{f(x)}{\alpha
+\beta }\left[ \alpha \frac{h(x)-h(a)}{x-a}+\beta \frac{h(b)-h(x)}{b-x}%
\right] \\
&-\frac{1}{\alpha +\beta }\left[ \frac{\alpha }{x-a}\int_{a}^{x}f(qt)%
\mathrm{D}_{q}h(t)\mathrm{d}_{q}t+\frac{\beta }{b-x}\int_{x}^{b}f(qt)\mathrm{%
D}_{q}h(t)\mathrm{d}_{q}t\right],
\end{align*}%
where%
\begin{equation*}
P\left( x,t\right) =\left\{
\begin{array}{l}
\frac{\alpha }{\alpha +\beta }\left( \frac{h(t)-h(a)}{x-a}\right), \text{ \
\ }a\leq t<x, \medskip\\
\frac{-\beta }{\alpha +\beta }\left( \frac{h(b)-h(t)}{b-x}\right), \text{ \ \ }%
x\leq t<b.%
\end{array}%
\right.
\end{equation*}%
Here, for $s,t\in q^{\mathbb{Z}}\cup \{0\}$ with $t\geq s$, we use the
definitions
\begin{equation*}
(\mathrm{D}_{q}f)(t)=\frac{f(qt)-f(t)}{(q-1)t}\ \ \mbox{and}\ \
\int_{s}^{t}f(\eta )\mathrm{d}_{q}\eta =(q-1)\sum\limits_{\ell =\log
_{q}(s)}^{\log _{q}(t/q)}f(q^{\ell })q^{\ell },
\end{equation*}%
by adopting the convention that $\log _{q}(0):=-\infty $ and $\log
_{q}(\infty ):=\infty $ $($see \cite{kc}$)$.
\end{corollary}

\begin{theorem}
\label{th5}Let $a,b,s,t\in \mathbb{T},$ $\ a<b$ and $f,h:[a,b]\rightarrow
\mathbb{R}$ be differentiable. Then for all $x\in \lbrack a,b]$, we have%
\begin{align*}
&\left\vert \frac{f(x)}{\alpha +\beta }\left[ \alpha \frac{h(x)-h(a)}{x-a}%
+\beta \frac{h(b)-h(x)}{b-x}\right] \right. \\
&\left. -\frac{1}{\alpha +\beta }\left[ \frac{\alpha }{x-a}%
\int_{a}^{x}h^{\Delta }(t)f(\sigma \left( t\right) )\Delta t+\frac{\beta }{%
b-x}\int_{x}^{b}h^{\Delta }(t)f(\sigma \left( t\right) )\Delta t\right]
\right\vert \\
\leq &\frac{M}{\alpha +\beta }\int_{a}^{b}\left\vert P(x,t)\right\vert
\Delta t,
\end{align*}%
where $\alpha, \beta\in \mathbb{R}$ are nonnegative and not both zero,
\begin{equation*}
P\left( x,t\right) =\left\{
\begin{array}{l}
\frac{\alpha }{\alpha +\beta }\left( \frac{h(t)-h(a)}{x-a}\right), \text{ \
\ }a\leq t<x, \medskip\\
\frac{-\beta }{\alpha +\beta }\left( \frac{h(b)-h(t)}{b-x}\right), \text{ \ \ }%
x\leq t<b%
\end{array}%
\right.
\end{equation*}
and%
\begin{equation*}
M=\underset{a<t<b}{\sup }\left\vert f^{\Delta }(t)\right\vert <\infty .
\end{equation*}
\end{theorem}

\begin{proof}
The proof of Theorem \ref{th5} can be done easily from Lemma \ref{lm1} by
using the properties of modulus and the definition of $h_{2}(\cdot ,\cdot )$.
\end{proof}

\begin{remark}
In the case of $\alpha=x-a$ and $\beta=b-x$ in Theorem \ref{th5}, we get
\begin{align*}
 \left\vert  \frac{h(b)-h(a)}{b-a}f(x) -\frac{1}{b-a }\int_{a}^{b}h^{\Delta }(t)f(\sigma \left( t\right) )\Delta t\right\vert
 \leq  \frac{M}{b-a }\left[\int_{a}^{x} |h(t)-h(a)| \Delta t+\int_{x}^{b}  |h(b)-h(t)| \Delta t\right],
\end{align*}%
which is the weighted version of $(\ref{1})$.
\end{remark}

\begin{corollary}
In the case of $\mathbb{T}=\mathbb{R}$ in Theorem \ref{th5}, we have%
\begin{align*}
&\left\vert \frac{f(x)}{\alpha +\beta }\left[ \alpha \frac{h(x)-h(a)}{x-a}%
+\beta \frac{h(b)-h(x)}{b-x}\right] -\frac{1}{\alpha +\beta }\left[ \frac{%
\alpha }{x-a}\int_{a}^{x}h^{\prime }(t)f(t)dt+\frac{\beta }{b-x}%
\int_{x}^{b}h^{\prime }(t)f(t)dt\right] \right\vert \\
\leq &\frac{M}{\alpha +\beta }\int_{a}^{b}\left\vert P(x,t)\right\vert dt,
\end{align*}%
where%
\begin{equation*}
P\left( x,t\right) =\left\{
\begin{array}{l}
\frac{\alpha }{\alpha +\beta }\left( \frac{h(t)-h(a)}{x-a}\right), \text{ \
\ }a\leq t<x, \medskip\\
\frac{-\beta }{\alpha +\beta }\left( \frac{h(b)-h(t)}{b-x}\right), \text{ \ \ }%
x\leq t<b%
\end{array}%
\right.
\end{equation*}
and%
\begin{equation*}
M=\underset{a<t<b}{\sup }\left\vert f^{\prime }(t)\right\vert <\infty .
\end{equation*}
\end{corollary}

\begin{corollary}
In the case of $\mathbb{T}=\mathbb{Z}$ in Theorem \ref{th5}, we have%
\begin{align*}
&\left\vert \frac{f(x)}{\alpha +\beta }\left[ \alpha \frac{h(x)-h(a)}{x-a}%
+\beta \frac{h(b)-h(x)}{b-x}\right] -\frac{1}{\alpha +\beta }\left[ \frac{%
\alpha }{x-a}\sum\limits_{t=a}^{x-1}f(t+1)\Delta h(t)+\frac{\beta }{b-x}%
\sum\limits_{t=x}^{b-1}f(t+1)\Delta h(t)\right] \right\vert \\
\leq &\frac{M}{\alpha +\beta }\sum\limits_{t=a}^{b-1}\left\vert
P(x,t)\right\vert ,
\end{align*}%
where%
\begin{equation*}
P\left( x,t\right) =\left\{
\begin{array}{l}
\frac{\alpha }{\alpha +\beta }\left( \frac{h(t)-h(a)}{x-a}\right) ,\text{ }%
a\leq t<x-1, \medskip\\
\frac{-\beta }{\alpha +\beta }\left( \frac{h(b)-h(t)}{b-x}\right) ,\text{ }%
x\leq t<b-1.%
\end{array}%
\right.
\end{equation*}
and%
\begin{equation*}
M=\underset{a<t<b}{\sup }\left\vert \Delta f(t)\right\vert <\infty .
\end{equation*}
\end{corollary}

\begin{corollary}
In the case of $\mathbb{T}=q^{\mathbb{Z}}\cup \{0\}\ (q>1)$ in Theorem \ref%
{th5}, we have%
\begin{align*}
&\left\vert \frac{f(x)}{\alpha +\beta }\left[ \alpha \frac{h(x)-h(a)}{x-a}%
+\beta \frac{h(b)-h(x)}{b-x}\right] -\frac{1}{\alpha +\beta }\left[ \frac{%
\alpha }{x-a}\int_{a}^{x}f(qt)\mathrm{D}_{q}h(t)\mathrm{d}_{q}t+\frac{\beta
}{b-x}\int_{x}^{b}f(qt)\mathrm{D}_{q}h(t)\mathrm{d}_{q}t\right] \right\vert
\\
\leq &\frac{M}{\alpha +\beta }\int_{a}^{b}\left\vert P(x,t)\right\vert
\mathrm{d}_{q}t,
\end{align*}%
where%
\begin{equation*}
P\left( x,t\right) =\left\{
\begin{array}{l}
\frac{\alpha }{\alpha +\beta }\left( \frac{h(t)-h(a)}{x-a}\right), \text{ \
\ }a\leq t<x, \medskip\\
\frac{-\beta }{\alpha +\beta }\left( \frac{h(b)-h(t)}{b-x}\right), \text{ \ \ }%
x\leq t<b%
\end{array}%
\right.
\end{equation*}
and%
\begin{equation*}
M=\underset{a<t<b}{\sup }\left\vert (\mathrm{D}_{q}f)(t)\right\vert <\infty .
\end{equation*}
\end{corollary}

\subsection{Weighted Ostrowski type inequalities for two functions on time scales}

 Then, we derive two weighted Ostrowski type inequalities for two functions on time scales.

\begin{theorem}
\label{th6}Let $a,b,s,t\in \mathbb{T},$ $\ a<b$ and $f,g,h:[a,b]\rightarrow
\mathbb{R}$ be differentiable. Then for all $x\in \lbrack a,b]$, we have%
\begin{align}
&\Bigg\vert \frac{f(x)g(x)}{\alpha +\beta }\left[ \alpha \frac{h(x)-h(a)}{%
x-a}+\beta \frac{h(b)-h(x)}{b-x}\right]   \notag \\
&-\frac{1}{2\left( \alpha +\beta \right) }\left\{ g(x)\left[ \frac{\alpha }{%
x-a}\int_{a}^{x}h^{\Delta }(t)f(\sigma \left( t\right) )\Delta t+\frac{\beta
}{b-x}\int_{x}^{b}h^{\Delta }(t)f(\sigma \left( t\right) )\Delta t\right]
\right.  \notag \\
&\left. \left. +f(x)\left[ \frac{\alpha }{x-a}\int_{a}^{x}h^{\Delta
}(t)g(\sigma \left( t\right) )\Delta t+\frac{\beta }{b-x}\int_{x}^{b}h^{%
\Delta }(t)g(\sigma \left( t\right) )\Delta t\right] \right\} \right\vert
\notag \\
\leq &\frac{M_{1}\left\vert g\left( x\right) \right\vert +M_{2}\left\vert
f\left( x\right) \right\vert }{2\left( \alpha +\beta \right) }\left[
\int_{a}^{b}\left\vert P(x,t)\right\vert \Delta t\right]  \label{6}
\end{align}%
and%
\begin{align}
&\left\vert f(x)g(x)\left[ \alpha \frac{h(x)-h(a)}{x-a}+\beta \frac{%
h(b)-h(x)}{b-x}\right] ^{2}-\left[ \alpha \frac{h(x)-h(a)}{x-a}+\beta \frac{h(b)-h(x)}{b-x%
}\right]\right.  \notag \\
&\times\left\{ f(x) \left[ \frac{\alpha }{x-a}\int_{a}^{x}h^{\Delta }(t)g(\sigma \left(
t\right) )\Delta t+\frac{\beta }{b-x}\int_{x}^{b}h^{\Delta }(t)g(\sigma
\left( t\right) )\Delta t\right] \right.  \notag \\
&\left. +g(x)\left[ \frac{\alpha }{x-a}\int_{a}^{x}h^{\Delta }(t)f(\sigma \left(
t\right) )\Delta t+\frac{\beta }{b-x}\int_{x}^{b}h^{\Delta }(t)f(\sigma
\left( t\right) )\Delta t\right] \right\}  \notag \\
&+\left[ \frac{\alpha }{x-a}\int_{a}^{x}h^{\Delta }(t)f(\sigma \left(
t\right) )\Delta t+\frac{\beta }{b-x}\int_{x}^{b}h^{\Delta }(t)f(\sigma
\left( t\right) )\Delta t\right]  \notag \\
&\left. \times \left[ \frac{\alpha }{x-a}\int_{a}^{x}h^{\Delta }(t)g(\sigma
\left( t\right) )\Delta t+\frac{\beta }{b-x}\int_{x}^{b}h^{\Delta
}(t)g(\sigma \left( t\right) )\Delta t\right] \right\vert  \notag \\
\leq &\left( \alpha +\beta \right) ^{2}\left( \int_{a}^{b}\left\vert
P(x,t)\right\vert \Delta t\right) ^{2},  \label{7}
\end{align}%
where $\alpha, \beta\in \mathbb{R}$ are nonnegative and not both zero,
\begin{equation*}
P\left( x,t\right) =\left\{
\begin{array}{l}
\frac{\alpha }{\alpha +\beta }\left( \frac{h(t)-h(a)}{x-a}\right), \text{ \
\ }a\leq t<x, \medskip\\
\frac{-\beta }{\alpha +\beta }\left( \frac{h(b)-h(t)}{b-x}\right), \text{ \ \ }%
x\leq t<b%
\end{array}%
\right.
\end{equation*}
and%
\begin{equation*}
M_{1}=\underset{a<t<b}{\sup }\left\vert f^{\Delta }(t)\right\vert <\infty
\text{ \ and \ }M_{2}=\underset{a<t<b}{\sup }\left\vert g^{\Delta
}(t)\right\vert <\infty .
\end{equation*}
\end{theorem}

\begin{proof}
We have%
\begin{align}
&\frac{f(x)}{\alpha +\beta }\left[ \alpha \frac{h(x)-h(a)}{x-a}+\beta \frac{%
h(b)-h(x)}{b-x}\right] \notag \\&-\frac{1}{\alpha +\beta }\left[ \frac{\alpha }{x-a}%
\int_{a}^{x}h^{\Delta }(t)f(\sigma \left( t\right) )\Delta t+\frac{\beta }{%
b-x}\int_{x}^{b}h^{\Delta }(t)f(\sigma \left( t\right) )\Delta t\right]
\notag \\
=&\int_{a}^{b}P(x,t)f^{\Delta }(t)\Delta t  \label{8}
\end{align}
and
\begin{align}
&\frac{g(x)}{\alpha +\beta }\left[ \alpha \frac{h(x)-h(a)}{x-a}+\beta \frac{%
h(b)-h(x)}{b-x}\right] \notag \\&-\frac{1}{\alpha +\beta }\left[ \frac{\alpha }{x-a}%
\int_{a}^{x}h^{\Delta }(t)g(\sigma \left( t\right) )\Delta t+\frac{\beta }{%
b-x}\int_{x}^{b}h^{\Delta }(t)g(\sigma \left( t\right) )\Delta t\right]
\notag \\
=&\int_{a}^{b}P(x,t)g^{\Delta }(t)\Delta t.  \label{9}
\end{align}

Multiplying (\ref{8}) by $g(x)$ and (\ref{9}) by $f(x)$, adding the
resultant identities, we have%
\begin{align*}
&\frac{f(x)g(x)}{\alpha +\beta }\left[ \alpha \frac{h(x)-h(a)}{x-a}+\beta
\frac{h(b)-h(x)}{b-x}\right] \\
&-\frac{1}{2\left( \alpha +\beta \right) }\left\{ g(x)\left[ \frac{\alpha }{%
x-a}\int_{a}^{x}h^{\Delta }(t)f(\sigma \left( t\right) )\Delta t+\frac{\beta
}{b-x}\int_{x}^{b}h^{\Delta }(t)f(\sigma \left( t\right) )\Delta t\right]
\right. \\
&\left. +f(x)\left[ \frac{\alpha }{x-a}\int_{a}^{x}h^{\Delta }(t)g(\sigma
\left( t\right) )\Delta t+\frac{\beta }{b-x}\int_{x}^{b}h^{\Delta
}(t)g(\sigma \left( t\right) )\Delta t\right] \right\} \\
=&\frac{1}{2}\left[ g\left( x\right) \int_{a}^{b}P(x,t)f^{\Delta }(t)\Delta
t+f\left( x\right) \int_{a}^{b}P(x,t)g^{\Delta }(t)\Delta t\right].
\end{align*}
Using the properties of modulus, we get%
\begin{align*}
&\left\vert \frac{f(x)g(x)}{\alpha +\beta }\left[ \alpha \frac{h(x)-h(a)}{%
x-a}+\beta \frac{h(b)-h(x)}{b-x}\right] \right. \\
&-\frac{1}{2\left( \alpha +\beta \right) }\left\{ g(x)\left[ \frac{\alpha }{%
x-a}\int_{a}^{x}h^{\Delta }(t)f(\sigma \left( t\right) )\Delta t+\frac{\beta
}{b-x}\int_{x}^{b}h^{\Delta }(t)f(\sigma \left( t\right) )\Delta t\right]
\right. \\
&\left. \left. +f(x)\left[ \frac{\alpha }{x-a}\int_{a}^{x}h^{\Delta
}(t)g(\sigma \left( t\right) )\Delta t+\frac{\beta }{b-x}\int_{x}^{b}h^{%
\Delta }(t)g(\sigma \left( t\right) )\Delta t\right] \right\} \right\vert \\
\leq &\frac{1}{2}\left[ \left\vert g\left( x\right) \right\vert
\int_{a}^{b}\left\vert P(x,t)\right\vert \left\vert f^{\Delta
}(t)\right\vert \Delta t+\left\vert f\left( x\right) \right\vert
\int_{a}^{b}\left\vert P(x,t)\right\vert \left\vert g^{\Delta
}(t)\right\vert \Delta t\right] \\
\leq &\frac{M_{1}\left\vert g\left( x\right) \right\vert +M_{2}\left\vert
f\left( x\right) \right\vert }{2\left( \alpha +\beta \right) }\left[
\int_{a}^{b}\left\vert P(x,t)\right\vert \Delta t\right].
\end{align*}%
This completes the proof of the inequality (\ref{6}).

Multiplying the left sides and right sides of (\ref{8}) and (\ref{9}), we get%
\begin{align*}
&f(x)g(x)\left[ \alpha \frac{h(x)-h(a)}{x-a}+\beta \frac{h(b)-h(x)}{b-x}%
\right] ^{2} \\
&-\left[ \alpha \frac{h(x)-h(a)}{x-a}+\beta \frac{h(b)-h(x)}{b-x}\right]
\left\{ f(x)\left[ \frac{\alpha }{x-a}\int_{a}^{x}h^{\Delta }(t)g(\sigma
\left( t\right) )\Delta t+\frac{\beta }{b-x}\int_{x}^{b}h^{\Delta
}(t)g(\sigma \left( t\right) )\Delta t\right] \right. \\
&\left. +g(x)\left[ \frac{\alpha }{x-a}\int_{a}^{x}h^{\Delta }(t)f(\sigma
\left( t\right) )\Delta t+\frac{\beta }{b-x}\int_{x}^{b}h^{\Delta
}(t)f(\sigma \left( t\right) )\Delta t\right] \right\} \\
&+\left[\frac{\alpha }{x-a}\int_{a}^{x}h^{\Delta }(t)f(\sigma \left(
t\right) )\Delta t+\frac{\beta }{b-x}\int_{x}^{b}h^{\Delta }(t)f(\sigma
\left( t\right) )\Delta t\right] \\
&\times \left[ \frac{\alpha }{x-a}\int_{a}^{x}h^{\Delta }(t)g(\sigma \left(
t\right) )\Delta t+\frac{\beta }{b-x}\int_{x}^{b}h^{\Delta }(t)g(\sigma
\left( t\right) )\Delta t\right] \\
=&\left( \alpha +\beta \right) ^{2}\left( \int_{a}^{b}P(x,t)f^{\Delta
}(t)\Delta t\right) \left( \int_{a}^{b}P(x,t)g^{\Delta }(t)\Delta t\right).
\end{align*}%
Using the properties of modulus, we can easily obtain (\ref{7}).
\end{proof}

\begin{corollary}\label{co7}
In the case of $\mathbb{T}=\mathbb{R}$ in Theorem \ref{th6}, we have%
\begin{align*}
&\left\vert \frac{f(x)g(x)}{\alpha +\beta }\left[ \alpha \frac{h(x)-h(a)}{%
x-a}+\beta \frac{h(b)-h(x)}{b-x}\right] \right. \\
&-\frac{1}{2\left( \alpha +\beta \right) }\left\{ g(x)\left[ \frac{\alpha }{%
x-a}\int_{a}^{x}h^{\prime }(t)f(t)dt+\frac{\beta }{b-x}\int_{x}^{b}h^{\prime
}(t)f(t)dt\right] \right. \\
&\left. \left. +f(x)\left[ \frac{\alpha }{x-a}\int_{a}^{x}h^{\prime
}(t)g(t)dt+\frac{\beta }{b-x}\int_{x}^{b}h^{\prime }(t)g(t)dt\right]
\right\} \right\vert \\
\leq &\frac{M_{1}\left\vert g\left( x\right) \right\vert +M_{2}\left\vert
f\left( x\right) \right\vert }{2\left( \alpha +\beta \right) }\left[
\int_{a}^{b}\left\vert P(x,t)\right\vert \Delta t\right]
\end{align*}%
and%
\begin{align*}
&\left\vert f(x)g(x)\left[ \alpha \frac{h(x)-h(a)}{x-a}+\beta \frac{%
h(b)-h(x)}{b-x}\right] ^{2}\right. \\
&-\left[ \alpha \frac{h(x)-h(a)}{x-a}+\beta \frac{h(b)-h(x)}{b-x}\right]
\left\{ f(x)\left[ \frac{\alpha }{x-a}\int_{a}^{x}h^{\prime }(t)g(t)dt+\frac{%
\beta }{b-x}\int_{x}^{b}h^{\prime }(t)g(t)dt\right] \right. \\
&\left. +g(x)\left[ \frac{\alpha }{x-a}\int_{a}^{x}h^{\prime }(t)f(t)dt+%
\frac{\beta }{b-x}\int_{x}^{b}h^{\prime }(t)f(t)dt\right] \right\} \\
&\left. +\left[ \frac{\alpha }{x-a}\int_{a}^{x}h^{\prime }(t)f(t)dt+\frac{%
\beta }{b-x}\int_{x}^{b}h^{\prime }(t)f(t)dt\right] \left[ \frac{\alpha }{x-a%
}\int_{a}^{x}h^{\prime }(t)g(t)dt+\frac{\beta }{b-x}\int_{x}^{b}h^{\prime
}(t)g(t)dt\right] \right\vert \\
\leq &\left( \alpha +\beta \right) ^{2}\left( \int_{a}^{b}\left\vert
P(x,t)\right\vert \Delta t\right) ^{2},
\end{align*}%
where%
\begin{equation*}
P\left( x,t\right) =\left\{
\begin{array}{l}
\frac{\alpha }{\alpha +\beta }\left( \frac{h(t)-h(a)}{x-a}\right), \text{ \
\ }a\leq t<x, \medskip\\
\frac{-\beta }{\alpha +\beta }\left( \frac{h(b)-h(t)}{b-x}\right), \text{ \ \ }%
x\leq t<b%
\end{array}%
\right.
\end{equation*}
and%
\begin{equation*}
M_{1}=\underset{a<t<b}{\sup }\left\vert f^{\prime }(t)\right\vert <\infty
\text{ \ and \ }M_{2}=\underset{a<t<b}{\sup }\left\vert g^{\prime
}(t)\right\vert <\infty .
\end{equation*}
\end{corollary}

\begin{remark}
In the case of $h(t)=t$, $\alpha=x-a$ and $\beta=b-x$ in Corollary \ref{co7}, we get the results given in \cite{p2003} $($for $h=0$$)$ and \cite{p2006} $($for $n=1$$)$.
\end{remark}

\begin{corollary}\label{co8}
In the case of $\mathbb{T}=\mathbb{Z}$ in Theorem \ref{th6}, we have%
\begin{align*}
&\left\vert \frac{f(x)g(x)}{\alpha +\beta }\left[ \alpha \frac{h(x)-h(a)}{%
x-a}+\beta \frac{h(b)-h(x)}{b-x}\right] \right. \\
&-\frac{1}{2\left( \alpha +\beta \right) }\left\{ g(x)\left[ \frac{\alpha }{%
x-a}\sum\limits_{t=a}^{x-1}f(t+1)\Delta h(t)+\frac{\beta }{b-x}%
\sum\limits_{t=x}^{b-1}f(t+1)\Delta h(t)\right] \right. \\
&\left. \left. +f(x)\left[ \frac{\alpha }{x-a}\sum\limits_{t=a}^{x-1}g(t+1)%
\Delta h(t)+\frac{\beta }{b-x}\sum\limits_{t=x}^{b-1}g(t+1)\Delta h(t)\right]
\right\} \right\vert \\
\leq &\frac{M_{1}\left\vert g\left( x\right) \right\vert +M_{2}\left\vert
f\left( x\right) \right\vert }{2\left( \alpha +\beta \right) }\left[
\sum\limits_{t=a}^{b-1}\left\vert P(x,t)\right\vert \right]
\end{align*}%
and%
\begin{align*}
&\left\vert f(x)g(x)\left[ \alpha \frac{h(x)-h(a)}{x-a}+\beta \frac{%
h(b)-h(x)}{b-x}\right] ^{2}\right. \\
&-\left[ \alpha \frac{h(x)-h(a)}{x-a}+\beta \frac{h(b)-h(x)}{b-x}\right]
\left\{ f(x)\left[ \frac{\alpha }{x-a}\sum\limits_{t=a}^{x-1}g(t+1)\Delta
h(t)+\frac{\beta }{b-x}\sum\limits_{t=x}^{b-1}g(t+1)\Delta h(t)\right]
\right. \\
&\left. +g(x)\left[ \frac{\alpha }{x-a}\sum\limits_{t=a}^{x-1}f(t+1)\Delta
h(t)+\frac{\beta }{b-x}\sum\limits_{t=x}^{b-1}f(t+1)\Delta h(t)\right]
\right\} \\
&+\left[ \frac{\alpha }{x-a}\sum\limits_{t=a}^{x-1}f(t+1)\Delta h(t)+\frac{%
\beta }{b-x}\sum\limits_{t=x}^{b-1}f(t+1)\Delta h(t)\right] \\
&\left. \times \left[ \frac{\alpha }{x-a}\sum\limits_{t=a}^{x-1}g(t+1)%
\Delta h(t)+\frac{\beta }{b-x}\sum\limits_{t=x}^{b-1}g(t+1)\Delta
h(t)\right] \right\vert \\
\leq &\left( \alpha +\beta \right) ^{2}\left(
\sum\limits_{t=a}^{b-1}\left\vert P(x,t)\right\vert \right) ^{2},
\end{align*}%
where%
\begin{equation*}
P\left( x,t\right) =\left\{
\begin{array}{l}
\frac{\alpha }{\alpha +\beta }\left( \frac{h(t)-h(a)}{x-a}\right) ,\text{ }%
a\leq t<x-1, \medskip\\
\frac{-\beta }{\alpha +\beta }\left( \frac{h(b)-h(t)}{b-x}\right) ,\text{ }%
x\leq t<b-1%
\end{array}%
\right.
\end{equation*}
and%
\begin{equation*}
M_{1}=\underset{a<t<b}{\sup }\left\vert \Delta f(t)\right\vert <\infty \text{
\ and \ }M_{2}=\underset{a<t<b}{\sup }\left\vert \Delta g(t)\right\vert
<\infty .
\end{equation*}
\end{corollary}

\begin{remark}
In the case of $h(t)=t$, $\alpha=x-a$ and $\beta=b-x$ in Corollary \ref{co8}, we get the results given in \cite[Theorem 1 and Theorem 2]{p2004}.
\end{remark}

\begin{corollary}
In the case of $\mathbb{T}=q^{\mathbb{Z}}\cup \{0\}\ (q>1)$ in Theorem \ref%
{th6}, we have%
\begin{align*}
&\left\vert \frac{f(x)g(x)}{\alpha +\beta }\left[ \alpha \frac{h(x)-h(a)}{%
x-a}+\beta \frac{h(b)-h(x)}{b-x}\right] \right. \\
&-\frac{1}{2\left( \alpha +\beta \right) }\left\{ g(x)\left[ \frac{\alpha }{%
x-a}\int_{a}^{x}f(qt)\mathrm{D}_{q}h(t)\mathrm{d}_{q}t+\frac{\beta }{b-x}%
\int_{x}^{b}f(qt)\mathrm{D}_{q}h(t)\mathrm{d}_{q}t\right] \right. \\
&\left. \left. +f(x)\left[ \frac{\alpha }{x-a}\int_{a}^{x}g(qt)\mathrm{D}%
_{q}h(t)\mathrm{d}_{q}t+\frac{\beta }{b-x}\int_{x}^{b}g(qt)\mathrm{D}_{q}h(t)%
\mathrm{d}_{q}t\right] \right\} \right\vert \\
\leq &\frac{M_{1}\left\vert g\left( x\right) \right\vert +M_{2}\left\vert
f\left( x\right) \right\vert }{2\left( \alpha +\beta \right) }\left[
\int_{a}^{b}\left\vert P(x,t)\right\vert \mathrm{d}_{q}t\right]
\end{align*}
and%
\begin{align*}
&\left\vert f(x)g(x)\left[ \alpha \frac{h(x)-h(a)}{x-a}+\beta \frac{%
h(b)-h(x)}{b-x}\right] ^{2}\right. \\
&-\left[ \alpha \frac{h(x)-h(a)}{x-a}+\beta \frac{h(b)-h(x)}{b-x}\right]
\left\{ f(x)\left[ \frac{\alpha }{x-a}\int_{a}^{x}g(qt)\mathrm{D}_{q}h(t)%
\mathrm{d}_{q}t+\frac{\beta }{b-x}\int_{x}^{b}g(qt)\mathrm{D}_{q}h(t)\mathrm{%
d}_{q}t\right] \right. \\
&\left. +g(x)\left[ \frac{\alpha }{x-a}\int_{a}^{x}f(qt)\mathrm{D}_{q}h(t)%
\mathrm{d}_{q}t+\frac{\beta }{b-x}\int_{x}^{b}f(qt)\mathrm{D}_{q}h(t)\mathrm{%
d}_{q}t\right] \right\} \\
&+\left[ \frac{\alpha }{x-a}\int_{a}^{x}f(qt)\mathrm{D}_{q}h(t)\mathrm{d}%
_{q}t+\frac{\beta }{b-x}\int_{x}^{b}f(qt)\mathrm{D}_{q}h(t)\mathrm{d}%
_{q}t\right] \\
&\left. \times \left[ \frac{\alpha }{x-a}\int_{a}^{x}g(qt)\mathrm{D}_{q}h(t)%
\mathrm{d}_{q}t+\frac{\beta }{b-x}\int_{x}^{b}g(qt)\mathrm{D}_{q}h(t)\mathrm{%
d}_{q}t\right] \right\vert \\
\leq &\left( \alpha +\beta \right) ^{2}\left( \int_{a}^{b}\left\vert
P(x,t)\right\vert \mathrm{d}_{q}t\right) ^{2},
\end{align*}%
\newline
where%
\begin{equation*}
P\left( x,t\right) =\left\{
\begin{array}{l}
\frac{\alpha }{\alpha +\beta }\left( \frac{h(t)-h(a)}{x-a}\right), \text{ \
\ }a\leq t<x, \medskip\\
\frac{-\beta }{\alpha +\beta }\left( \frac{h(b)-h(t)}{b-x}\right), \text{ \ \ }%
x\leq t<b%
\end{array}%
\right.
\end{equation*}
and%
\begin{equation*}
M_{1}=\underset{a<t<b}{\sup }\left\vert (\mathrm{D}_{q}f)(t)\right\vert
<\infty \text{ \ and \ }M_{2}=\underset{a<t<b}{\sup }\left\vert (\mathrm{D}%
_{q}g)(t)\right\vert <\infty .
\end{equation*}
\end{corollary}

\subsection{New weighted perturbed Ostrowski type inequalities on time scales} In this subsection, four  weighted perturbed Ostrowski type inequalities on
time scales are established.

\begin{theorem}
\label{th7} Let $a,b,s,t\in \mathbb{T},$ $\ a<b$ and $f,h:[a,b]\rightarrow
\mathbb{R}$ be differentiable. Then for all $x\in \lbrack a,b]$, we have%
\begin{align}
&\left\vert \frac{f(x)}{\alpha +\beta }\left[ \alpha \frac{h(x)-h(a)}{x-a}%
+\beta \frac{h(b)-h(x)}{b-x}\right]-\frac{f(b)-f(a)}{b-a}\left(\int_{a}^{b}P(x,t)\Delta t\right) \right.  \notag \\
&\left.-\frac{1}{\alpha +\beta }\left[ \frac{\alpha }{x-a}\int_{a}^{x}h^{\Delta
}(t)f(\sigma \left( t\right) )\Delta t+\frac{\beta }{b-x}\int_{x}^{b}h^{%
\Delta }(t)f(\sigma \left( t\right) )\Delta t\right]    \right\vert  \notag \\
\leq &\left( b-a\right) \left[ \frac{1}{b-a}\int_{a}^{b}P^{2}(x,t)\Delta
t-\left( \frac{1}{b-a}\int_{a}^{b}P(x,t)\Delta t\right) ^{2}\right] ^{\frac{1%
}{2}}  \notag \\
&\times \left[ \frac{1}{b-a}\int_{a}^{b}\left( f^{\Delta }\left( t\right)
\right) ^{2}\Delta t-\left(\frac{f(b)-f(a)}{b-a}\right) ^{2}\right] ^{\frac{1}{2}},\ \ f^{\Delta}\in L^2[a,b];\notag \\
\leq &\left( b-a\right) \left[ \frac{1}{b-a}\int_{a}^{b}P^{2}(x,t)\Delta
t-\left( \frac{1}{b-a}\int_{a}^{b}P(x,t)\Delta t\right) ^{2}\right] ^{\frac{1%
}{2}}  \frac{\Gamma-\gamma}{2}, \ \ \gamma\le f^{\Delta}(x)\le \Gamma, \ x\in [a,b], \label{10}
\end{align}
where $\alpha, \beta\in \mathbb{R}$ are nonnegative and not both zero,
\begin{equation*}
P\left( x,t\right) =\left\{
\begin{array}{l}
\frac{\alpha }{\alpha +\beta }\left( \frac{h(t)-h(a)}{x-a}\right), \text{ \
\ }a\leq t<x, \medskip\\
\frac{-\beta }{\alpha +\beta }\left( \frac{h(b)-h(t)}{b-x}\right), \text{ \ \ }%
x\leq t<b.
\end{array}%
\right.
\end{equation*}
\end{theorem}

\begin{proof}
We have
\begin{align}
&\frac{1}{b-a}\int_{a}^{b}P(x,t)f^{\Delta }\left( t\right) \Delta t-\left(
\frac{1}{b-a}\int_{a}^{b}P(x,t)\Delta t\right) \left( \frac{1}{b-a}%
\int_{a}^{b}f^{\Delta }\left( t\right) \Delta t\right)  \notag \\
=&\frac{1}{2\left( b-a\right) ^{2}}\int_{a}^{b}\int_{a}^{b}\left(
P(x,t)-P(x,s)\right) \left( f^{\Delta }\left( t\right) -f^{\Delta }\left(
s\right) \right) \Delta t\Delta s.  \label{11}
\end{align}%
From (\ref{2}), we also have%
\begin{align}
&\int_{a}^{b}P(x,t)f^{\Delta }(t)\Delta t  \notag \\
=&\frac{f(x)}{\alpha +\beta }\left[ \alpha \frac{h(x)-h(a)}{x-a}+\beta
\frac{h(b)-h(x)}{b-x}\right]\notag \\&
 -\frac{1}{\alpha +\beta }\left[ \frac{\alpha }{%
x-a}\int_{a}^{x}h^{\Delta }(t)f(\sigma \left( t\right) )\Delta t+\frac{\beta
}{b-x}\int_{x}^{b}h^{\Delta }(t)f(\sigma \left( t\right) )\Delta t\right]
\label{12}
\end{align}
and
\begin{equation}
\frac{1}{b-a}\int_{a}^{b}f^{\Delta }\left( t\right) \Delta t=\frac{f\left(
b\right) -f\left( a\right) }{b-a}.  \label{13}
\end{equation}%
Using the Cauchy-Schwartz inequality, we may write%
\begin{align}
\ \ \ \ \ \ &\left\vert \frac{1}{2\left( b-a\right) ^{2}}%
\int_{a}^{b}\int_{a}^{b}\left( P(x,t)-P(x,s)\right) \left( f^{\Delta }\left(
t\right) -f^{\Delta }\left( s\right) \right) \Delta t\Delta s\right\vert
\notag \\
\leq &\left( \frac{1}{2\left( b-a\right) ^{2}}\int_{a}^{b}\int%
\limits_{a}^{b}\left( P(x,t)-P(x,s)\right) ^{2}\Delta t\Delta s\right) ^{%
\frac{1}{2}}\left( \frac{1}{2\left( b-a\right) ^{2}}\int_{a}^{b}\int%
\limits_{a}^{b}\left( f^{\Delta }\left( t\right) -f^{\Delta }\left( s\right)
\right) ^{2}\Delta t\Delta s\right) ^{\frac{1}{2}}.  \label{14}
\end{align}

However%
\begin{equation}
\frac{1}{2\left( b-a\right) ^{2}}\int_{a}^{b}\left( P(x,t)-P(x,s)\right)
^{2}\Delta t\Delta s=\frac{1}{b-a}\int_{a}^{b}P^{2}(x,t)\Delta t-\left(
\frac{1}{b-a}\int_{a}^{b}P(x,t)\Delta t\right) ^{2}  \label{15}
\end{equation}%
and (see \cite[inequality (3.3)]{mpu1999})
\begin{align}
\frac{1}{2\left( b-a\right) ^{2}}\int_{a}^{b}\int_{a}^{b}\left( f^{\Delta
}\left( t\right) -f^{\Delta }\left( s\right) \right) ^{2}\Delta t\Delta s=
&\frac{1}{b-a}\int_{a}^{b}\left( f^{\Delta }\left( t\right) \right)
^{2}\Delta t-\left( \frac{1}{b-a}\int_{a}^{b}f^{\Delta }\left( t\right)
\Delta t\right) ^{2}\notag\\
=&\frac{1}{b-a}\int_{a}^{b}\left( f^{\Delta }\left( t\right) \right)
^{2}\Delta t-\left( \frac{f(b)-f(a)}{b-a} \right) ^{2}\notag\\
\le &\left[\frac{\Gamma-\gamma}{2}\right]^2,\ \text{where} \ \gamma\le f^{\Delta}(t)\le \Gamma, \ t\in [a,b].   \label{16}
\end{align}%
Using (\ref{11})-(\ref{16}), we can easily obtain the inequality (\ref{10}).
\end{proof}

\begin{corollary}
In the case of $\mathbb{T}=\mathbb{R}$ in Theorem \ref{th7}, we have%
\begin{align*}
&\left\vert \frac{f(x)}{\alpha +\beta }\left[ \alpha \frac{h(x)-h(a)}{x-a}%
+\beta \frac{h(b)-h(x)}{b-x}\right]-\frac{f(b)-f(a)}{b-a}\left(\int_{a}^{b}P(x,t)d t\right) \right. \\
&\left.-\frac{1}{\alpha +\beta }\left[ \frac{\alpha }{x-a}\int_{a}^{x}h^{\prime
}(t)f(t)dt+\frac{\beta }{b-x}\int_{x}^{b}h^{\prime }(t)f(t)dt\right]   \right\vert \\
\leq &\left( b-a\right) \left[ \frac{1}{b-a}\int_{a}^{b}P^{2}(x,t)dt-\left(
\frac{1}{b-a}\int_{a}^{b}P(x,t)dt\right) ^{2}\right] ^{\frac{1}{2}} \\
&\times\left[ \frac{1}{b-a}\int_{a}^{b}\left( f^{\prime }\left( t\right) \right)
^{2}dt-\left( \frac{f(b)-f(a)}{b-a}\right) ^{2}\right] ^{\frac{1}{2}}, \ \ f'\in L^2[a,b];\notag \\
\leq &\left( b-a\right) \left[ \frac{1}{b-a}\int_{a}^{b}P^{2}(x,t)dt-\left(
\frac{1}{b-a}\int_{a}^{b}P(x,t)dt\right) ^{2}\right] ^{\frac{1}{2}}\frac{\Gamma-\gamma}{2}, \ \ \gamma\le f'(x)\le \Gamma, \ x\in [a,b],
\end{align*}
where
\begin{equation*}
P\left( x,t\right) =\left\{
\begin{array}{l}
\frac{\alpha }{\alpha +\beta }\left( \frac{h(t)-h(a)}{x-a}\right), \text{ \
\ }a\leq t<x, \medskip\\
\frac{-\beta }{\alpha +\beta }\left( \frac{h(b)-h(t)}{b-x}\right), \text{ \ \ }%
x\leq t<b.%
\end{array}%
\right.
\end{equation*}
\end{corollary}

\begin{corollary}
In the case of $\mathbb{T}=\mathbb{Z}$ in Theorem \ref{th7}, we have%
\begin{align*}
&\left\vert \frac{f(x)}{\alpha +\beta }\left[ \alpha \frac{h(x)-h(a)}{x-a}%
+\beta \frac{h(b)-h(x)}{b-x}\right]-\frac{f(b)-f(a)}{b-a}\left(\sum\limits_{t=a}^{b-1}P(x,t)\right) \right. \\
&\left. -\frac{1}{\alpha +\beta }\left[ \frac{\alpha }{x-a}\sum%
\limits_{t=a}^{x-1}f(t+1)\Delta h(t)+\frac{\beta }{b-x}\sum%
\limits_{t=x}^{b-1}f(t+1)\Delta h(t)\right]  \right\vert \\
\leq &\left( b-a\right) \left[ \frac{1}{b-a}\sum%
\limits_{t=a}^{b-1}P^{2}(x,t)-\left( \frac{1}{b-a}\sum%
\limits_{t=a}^{b-1}P(x,t)\right) ^{2}\right] ^{\frac{1}{2}} \\
&\times \left[ \frac{1}{b-a}\sum\limits_{t=a}^{b-1}\left( \Delta f\left(
t\right) \right) ^{2}-\left( \frac{f(b)-f(a)}{b-a}\right) ^{2}\right] ^{%
\frac{1}{2}};\notag \\
\leq &\left( b-a\right) \left[ \frac{1}{b-a}\sum%
\limits_{t=a}^{b-1}P^{2}(x,t)-\left( \frac{1}{b-a}\sum%
\limits_{t=a}^{b-1}P(x,t)\right) ^{2}\right] ^{\frac{1}{2}}\frac{\Gamma-\gamma}{2}, \ \ \gamma\le \Delta f(t)\le \Gamma, \ t\in [a,b],
\end{align*}
where%
\begin{equation*}
P\left( x,t\right) =\left\{
\begin{array}{l}
\frac{\alpha }{\alpha +\beta }\left( \frac{h(t)-h(a)}{x-a}\right) ,\text{ }%
a\leq t<x-1, \medskip\\
\frac{-\beta }{\alpha +\beta }\left( \frac{h(b)-h(t)}{b-x}\right) ,\text{ }%
x\leq t<b-1.%
\end{array}%
\right.
\end{equation*}
\end{corollary}

\begin{corollary}
In the case of $\mathbb{T}=q^{\mathbb{Z}}\cup \{0\}\ (q>1)$ in Theorem \ref%
{th7}, we have%
\begin{align*}
&\left\vert \frac{f(x)}{\alpha +\beta }\left[ \alpha \frac{h(x)-h(a)}{x-a}%
+\beta \frac{h(b)-h(x)}{b-x}\right]-\frac{f\left( b\right) -f\left( a\right) }{b-a}\left( \frac{1}{b-a}%
\int_{a}^{b}P(x,t)\mathrm{d}_{q}t\right) \right. \\
&\left.-\frac{1}{\alpha +\beta }\left[ \frac{\alpha }{x-a}\int_{a}^{x}f(qt)%
\mathrm{D}_{q}h(t)\mathrm{d}_{q}t+\frac{\beta }{b-x}\int_{x}^{b}f(qt)\mathrm{%
D}_{q}h(t)\mathrm{d}_{q}t\right]   \right\vert \\
\leq &\left( b-a\right) \left[\frac{1}{b-a}\int_{a}^{b}P^{2}(x,t)\mathrm{d}%
_{q}t-\left( \frac{1}{b-a}\int_{a}^{b}P(x,t)\mathrm{d}_{q}t\right)
^{2}\right] ^{\frac{1}{2}} \\
&\times \left[\frac{1}{b-a}\int_{a}^{b}\left( \mathrm{D}_{q}f(t)\right)
^{2}\mathrm{d}_{q}t-\left( \frac{f(b)-f(a)}{b-a}\right) ^{2}\right] ^{\frac{1%
}{2}}, \ \ \mathrm{D}_{q}f(t)\in L^2[a,b];\notag \\
\leq &\left( b-a\right) \left[\frac{1}{b-a}\int_{a}^{b}P^{2}(x,t)\mathrm{d}%
_{q}t-\left( \frac{1}{b-a}\int_{a}^{b}P(x,t)\mathrm{d}_{q}t\right)
^{2}\right] ^{\frac{1}{2}}\frac{\Gamma-\gamma}{2}, \ \ \gamma\le \mathrm{D}_{q}f(t)\le \Gamma, \ t\in [a,b],
\end{align*}
where%
\begin{equation*}
P\left( x,t\right) =\left\{
\begin{array}{l}
\frac{\alpha }{\alpha +\beta }\left( \frac{h(t)-h(a)}{x-a}\right), \text{ \
\ }a\leq t<x, \medskip\\
\frac{-\beta }{\alpha +\beta }\left( \frac{h(b)-h(t)}{b-x}\right), \text{ \ \ }%
x\leq t<b.
\end{array}%
\right.
\end{equation*}
\end{corollary}

\begin{theorem}
\label{th8}Let $a,b,x,t\in \mathbb{T}\mathbf{,}$ $a<b$ and $f,h:\left[ a,b%
\right] \rightarrow \mathbb{R}$ be differentiable function such that there
exist constants $\gamma ,\Gamma \in \mathbb{R}$, with $\gamma \leq f^{\Delta
}\left( x\right) \leq \Gamma ,$ $x\in \left[ a,b\right] $. Then for all $%
x\in \left[ a,b\right] ,$ we have%
\begin{align}
&\left\vert \frac{f(x)}{\alpha +\beta }\left[ \alpha \frac{h(x)-h(a)}{x-a}%
+\beta \frac{h(b)-h(x)}{b-x}\right] -\frac{\gamma +\Gamma }{2}\left(
\int_{a}^{b}P(x,t)\Delta t\right) \right.  \notag \\
&\left. -\frac{1}{\alpha +\beta }\left[ \frac{\alpha }{x-a}%
\int_{a}^{x}h^{\Delta }(t)f(\sigma \left( t\right) )\Delta t+\frac{\beta }{%
b-x}\int_{x}^{b}h^{\Delta }(t)f(\sigma \left( t\right) )\Delta t\right]
\right\vert  \notag \\
\leq &\frac{\Gamma -\gamma }{2}\left( \int_{a}^{b}\left\vert
P(x,t)\right\vert \Delta t\right),  \label{17}
\end{align}
where $\alpha, \beta\in \mathbb{R}$ are nonnegative and not both zero,
\begin{equation*}
P\left( x,t\right) =\left\{
\begin{array}{l}
\frac{\alpha }{\alpha +\beta }\left( \frac{h(t)-h(a)}{x-a}\right), \text{ \
\ }a\leq t<x, \medskip\\
\frac{-\beta }{\alpha +\beta }\left( \frac{h(b)-h(t)}{b-x}\right), \text{ \ \ }%
x\leq t<b.
\end{array}%
\right.
\end{equation*}
\end{theorem}

\begin{proof}
From (\ref{2}), we may write%
\begin{align}
&\frac{f(x)}{\alpha +\beta }\left[ \alpha \frac{h(x)-h(a)}{x-a}+\beta \frac{%
h(b)-h(x)}{b-x}\right]  \notag \\
=&\frac{1}{\alpha +\beta }\left[ \frac{\alpha }{x-a}\int_{a}^{x}h^{\Delta
}(t)f(\sigma \left( t\right) )\Delta t+\frac{\beta }{b-x}\int_{x}^{b}h^{%
\Delta }(t)f(\sigma \left( t\right) )\Delta t\right] +\int_{a}^{b}P(x,t)f^{%
\Delta }(t)\Delta t.  \label{18}
\end{align}%
Let $C=\frac{\gamma +\Gamma }{2}.$ From (\ref{18}), we get%
\begin{align}
&\int_{a}^{b}P(x,t)\left( f^{\Delta }(t)-C\right) \Delta t  \notag \\
=&\frac{f(x)}{\alpha +\beta }\left[ \alpha \frac{h(x)-h(a)}{x-a}+\beta
\frac{h(b)-h(x)}{b-x}\right]-\frac{\gamma +\Gamma }{2}\left( \int_{a}^{b}P(x,t)\Delta t\right)  \notag \\
&-\frac{1}{\alpha +\beta }\left[ \frac{\alpha }{x-a}\int_{a}^{x}h^{\Delta
}(t)f(\sigma \left( t\right) )\Delta t+\frac{\beta }{b-x}\int_{x}^{b}h^{%
\Delta }(t)f(\sigma \left( t\right) )\Delta t\right].
\label{19}
\end{align}%
Using the properties of modulus, we get
\begin{equation}
\left\vert \int_{a}^{b}P(x,t)\left( f^{\Delta }(t)-C\right) \Delta
t\right\vert \leq \frac{\Gamma -\gamma }{2}\left( \int_{a}^{b}\left\vert
P(x,t)\right\vert \Delta t\right).  \label{20}
\end{equation}%
From (\ref{18})-(\ref{20}), we can easily get (\ref{17}).
\end{proof}

\begin{corollary}\label{co13}
In the case of $\mathbb{T}=\mathbb{R}$ in Theorem \ref{th8}, we have%
\begin{align*}
&\left\vert \frac{f(x)}{\alpha +\beta }\left[ \alpha \frac{h(x)-h(a)}{x-a}%
+\beta \frac{h(b)-h(x)}{b-x}\right] -\frac{\gamma +\Gamma }{2}\left(
\int_{a}^{b}P(x,t)dt\right) \right. \\
&\left. -\frac{1}{\alpha +\beta }\left[ \frac{\alpha }{x-a}%
\int_{a}^{x}h^{\prime }(t)f(t)dt+\frac{\beta }{b-x}\int_{x}^{b}h^{\prime
}(t)f(t)dt\right] \right\vert \\
\leq &\frac{\Gamma -\gamma }{2}\left( \int_{a}^{b}\left\vert
P(x,t)\right\vert dt\right),
\end{align*}
where%
\begin{equation*}
P\left( x,t\right) =\left\{
\begin{array}{l}
\frac{\alpha }{\alpha +\beta }\left( \frac{h(t)-h(a)}{x-a}\right), \text{ \
\ }a\leq t<x, \medskip\\
\frac{-\beta }{\alpha +\beta }\left( \frac{h(b)-h(t)}{b-x}\right), \text{ \ \ }%
x\leq t<b.%
\end{array}%
\right.
\end{equation*}
\end{corollary}

\begin{remark}
In the case of $h(t)=t$, $\alpha=x-a$ and $\beta=b-x$ in Corollary \ref{co13}, we recapture the result given in \cite[Corollary 1]{u2004}.
\end{remark}

\begin{corollary}
In the case of $\mathbb{T}=\mathbb{Z}$ in Theorem \ref{th8}, we have%
\begin{align*}
&\left\vert \frac{f(x)}{\alpha +\beta }\left[ \alpha \frac{h(x)-h(a)}{x-a}%
+\beta \frac{h(b)-h(x)}{b-x}\right] -\frac{\gamma +\Gamma }{2}\left(
\sum\limits_{t=a}^{b-1}P(x,t)\right) \right. \\
&\left. -\frac{1}{\alpha +\beta }\left[ \frac{\alpha }{x-a}%
\sum\limits_{t=a}^{x-1}f(t+1)\Delta h(t)+\frac{\beta }{b-x}%
\sum\limits_{t=x}^{b-1}f(t+1)\Delta h(t)\right] \right\vert \\
\leq &\frac{\Gamma -\gamma }{2}\left( \sum\limits_{t=a}^{b-1}\left\vert
P(x,t)\right\vert \right),
\end{align*}
where%
\begin{equation*}
P\left( x,t\right) =\left\{
\begin{array}{l}
\frac{\alpha }{\alpha +\beta }\left( \frac{h(t)-h(a)}{x-a}\right) ,\text{ }%
a\leq t<x-1, \medskip\\
\frac{-\beta }{\alpha +\beta }\left( \frac{h(b)-h(t)}{b-x}\right) ,\text{ }%
x\leq t<b-1.%
\end{array}%
\right.
\end{equation*}
\end{corollary}

\begin{corollary}
In the case of $\mathbb{T}=q^{\mathbb{Z}}\cup \{0\}\ (q>1)$ in Theorem \ref%
{th8}, we have%
\begin{align*}
&\left\vert \frac{f(x)}{\alpha +\beta }\left[ \alpha \frac{h(x)-h(a)}{x-a}%
+\beta \frac{h(b)-h(x)}{b-x}\right] -\frac{\gamma +\Gamma }{2}\left(
\int_{a}^{b}P(x,t)\mathrm{d}_{q}t\right) \right. \\
&\left. -\frac{1}{\alpha +\beta }\left[ \frac{\alpha }{x-a}\int_{a}^{x}f(qt)%
\mathrm{D}_{q}h(t)\mathrm{d}_{q}t+\frac{\beta }{b-x}\int_{x}^{b}f(qt)\mathrm{%
D}_{q}h(t)\mathrm{d}_{q}t\right] \right\vert \\
\leq &\frac{\Gamma -\gamma }{2}\left( \int_{a}^{b}\left\vert
P(x,t)\right\vert \mathrm{d}_{q}t\right),
\end{align*}
where%
\begin{equation*}
P\left( x,t\right) =\left\{
\begin{array}{l}
\frac{\alpha }{\alpha +\beta }\left( \frac{h(t)-h(a)}{x-a}\right), \text{ \
\ }a\leq t<x, \medskip\\
\frac{-\beta }{\alpha +\beta }\left( \frac{h(b)-h(t)}{b-x}\right), \text{ \ \ }%
x\leq t<b.%
\end{array}%
\right.
\end{equation*}
\end{corollary}

\section*{Acknowledgments}

This work was supported by the National Natural Science Foundation of China
(Grant No. 41174165), the Tianyuan Fund of Mathematics (Grant No. 11026211)
and the Natural Science Foundation of the Jiangsu Higher Education
Institutions (Grant No. 09KJB110005).

\end{document}